\documentclass[12pt,a4paper,psamsfonts]{amsart}
\usepackage{amssymb,amscd,amsxtra,calc}
\usepackage{mathrsfs}
\usepackage{cmmib57}
\usepackage{multirow}
\usepackage[all]{xy}
\usepackage[colorlinks,linkcolor=blue,anchorcolor=blue,citecolor=green]{hyperref}
\theoremstyle{plain}
    \newtheorem{thm}{Theorem}[section]

    \newtheorem{corollary}[thm]{Corollary}
    \newtheorem{example}[thm]{Example}
    \newtheorem{lemma}[thm]{Lemma}
    \newtheorem{proposition}[thm]{Proposition}
    
    \newtheorem{theorem}[thm]{Theorem}

\theoremstyle{definition}
    \newtheorem{definition}[thm]{Definition}

    \newtheorem{remark}[thm]{Remark}
\theoremstyle{remark}

    \newtheorem{setup}[thm]{}

\newcommand{\Q}{\mathbb{Q}}

\newcommand{\R}{\mathbb{R}}

\newcommand{\Z}{\mathbb{Z}}

\newcommand{\Gal}{\operatorname{Gal}}

\newcommand{\rank}{\operatorname{rank}}

\newcommand{\Supp}{\operatorname{Supp}}

\newcommand{\Codim}{\operatorname{codim}}

\newcommand{\alg}{\mathrm{alg}}
\newcommand{\reg}{\mathrm{reg}}

\begin{document}

\title[Characterizations of Toric Varieties via Polarized Endomorphisms]
{Characterizations of Toric Varieties via Polarized Endomorphisms}

\author{Sheng Meng, De-Qi Zhang}

\address
{
\textsc{Department of Mathematics} \endgraf
\textsc{National University of Singapore, 10 Lower Kent Ridge Road,
Singapore 119076
}}
\email{ms@u.nus.edu}
\email{matmeng@nus.edu.sg}
\address
{
\textsc{Department of Mathematics} \endgraf
\textsc{National University of Singapore, 10 Lower Kent Ridge Road,
Singapore 119076
}}
\email{matzdq@nus.edu.sg}

\begin{abstract}
Let $X$ be a normal projective variety and $f:X\to X$ a non-isomorphic polarized endomorphism.
We give two characterizations for $X$ to be a toric variety.

First we show that if $X$ is $\Q$-factorial and $G$-almost homogeneous for some linear algebraic group $G$ such that $f$ is $G$-equivariant, then $X$ is a toric variety.

Next we give a geometric characterization:
if $X$ is of Fano type and smooth in codimension 2 and
if there is an $f^{-1}$-invariant reduced divisor $D$ such that $f|_{X\backslash D}$ is quasi-\'etale and $K_X+D$ is $\Q$-Cartier, then $X$ admits a quasi-\'etale cover $\widetilde{X}$ such that $\widetilde{X}$ is a toric variety and $f$ lifts to $\widetilde{X}$. In particular, if $X$ is further assumed to be smooth, then $X$ is a toric variety.
\end{abstract}



\maketitle

\section{Introduction}

We work over an algebraically closed field $k$ of characteristic zero.
Let $X$ be a normal projective variety of dimension $n \ge 1$.

A surjective endomorphism $f : X \to X$ is {\it polarized},
if there is an ample Cartier divisor $H$ on $X$ such that we have $f^{\ast}H \sim qH$
(linear equivalence) for some integer $q>1$.
See \cite{Zsw} for some conjectures on polarised endomorphisms.

The variety $X$ is said to be {\it toric} or a {\it toric variety}
if $X$ contains an algebraic torus $T = (k^*)^n$ as an (affine) open dense subset such that
the natural multiplication action of $T$ on itself extends to an action on the whole variety $X$.
In this case, let $D:=X\backslash T$, which is a divisor; the pair $(X,D)$ is said to be a {\it toric pair}.

We observe that a toric variety (e.g. the projective space) has lots of symmetries.
The purpose of this short paper is to use symmetries to characterize toric pairs.

We first give a characterization of toric varieties via polarized endomorphisms and linear algebraic group actions.
Let $G$ be a linear algebraic group acting on $X$.
We say $X$ is {\it $G$-almost homogeneous} if there exists an open dense $G$-orbit in $X$.

\begin{theorem}\label{thm-torus}
Let $f:X\to X$ be a polarized endomorphism of a $G$-almost homogeneous normal projective variety $X$ with $G$ being a linear algebraic group.
Assume further the following conditions.
\begin{itemize}
\item[(i)] Let $U$ be the open dense $G$-orbit in $X$ and $D$ the codimension-$1$ part of $X \setminus U$.
The Weil divisor $K_X+D$ is $\Q$-Cartier.
\item[(ii)]
The endomorphism $f$ is $G$-equivariant in the sense: there is a surjective
homomorphism $\varphi: G \to G$ such that
$f\circ g=\varphi(g)\circ f$ for all $g$ in $G$.
\end{itemize}
Then $(X, D)$ is a toric pair.
\end{theorem}

Let $(X,\Delta)$ be a log pair.
The {\it complexity} $c = c(X,\Delta)$ of $(X,\Delta)$ is defined as
$$c:=\inf\{n+\dim_{\R}(\sum \R[S_i])-\sum a_i\,|\, \sum a_i S_i\le \Delta, a_i\ge 0, S_i\ge 0\};$$
see Definition \ref{def-comp}.
Brown, ${\rm M}^c$Kernan, Svaldi and Zong recently gave a geometric characterization of toric varieties involving the complexity; see \cite[Theorem 1.2]{BMSZ} or Theorem \ref{thm-bmsz}.
Their result is a special case of a conjecture of Shokurov, which is stated in the relative case (cf.~\cite{Sho}).
A simple version of their result shows that
{\it if $(X,\Delta)$ is a log canonical pair such that $\Delta$ is reduced, $-(K_X + \Delta)$ is nef, and
$c(X,\Delta)$ is non-positive, then $(X,\Delta)$ is a toric pair};
see Theorem \ref{thm-bmsz} and Remark \ref{rmk-comp}.

A finite surjective morphism $h : Y \to Z$ between normal varieties is {\it quasi-\'etale} if $h$ is \'etale
outside a codimension-$2$ subset of $Z$.
We show that the complexity condition of \cite[Theorem 1.2]{BMSZ} holds true in the following case.

\begin{theorem}\label{thm-pitrivial}
Let $X$ be a normal projective variety
which is smooth in codimension $2$, and $D\subset X$ a reduced divisor such that
\begin{itemize}
\item[(i)]
there is a Weil $\Q$-divisor $\Gamma$ such that the pair $(X, \Gamma)$ has only klt singularities;
\item[(ii)]
there is a polarized endomorphism $f:X\to X$ such that $D$ is $f^{-1}$-invariant and $f|_{X\backslash D}$ is quasi-\'etale;
\item[(iii)]
the algebraic fundamental group
$\pi_1^{\alg}(X_{\reg})$ of the smooth locus $X_{\reg}$ of $X$ is trivial
(this holds when $X$ is smooth and rationally connected); and
\item[(iv)]
the irregularity $q(X) := h^1(X, \mathcal{O}_X)$ is zero (this holds when $X$ is rationally connected).
\end{itemize}
Then the complexity $c(X,D)$ is non-positive.
\end{theorem}

\begin{remark} Let $X$ be an $n$-dimensional smooth Fano variety of Picard number one
and $D\subset X$ a reduced divisor.
Assume the existence of a non-isomorphic surjective endomorphism $f : X \to X$ such that  $D$ is $f^{-1}$-invariant and $f|_{X\backslash D}$ is \'etale.
Hwang and Nakayama show that $X$ is isomorphic to $\mathbb{P}^n$ and $D$ is a simple normal crossing divisor consisting of $n + 1$ hyperplanes;
see \cite[Theorem 2.1]{HN}.
In particular, $(X,D)$ is a toric pair.
Indeed, their argument shows that the complexity $c(X,D)$ is non-positive.
Our Theorem \ref{thm-pitrivial} follows their idea and tries to generalize their result to the singular case.
A key step of ours is to verify that $\hat{\Omega}_X^1(\log D)$
is free, i.e., isomorphic to $\mathcal{O}_X^{\oplus n}$;
see \cite[Proposition 2.3]{HN} and Theorem \ref{HN-thm}.
\end{remark}

As applications of Theorem \ref{thm-pitrivial}, we have the following characterizations
for toric pairs.
A well known conjecture asserts that projective spaces are the only
smooth projective Fano varieties of Picard number one admitting an endomorphism $f$ of degree $\ge 2$
(this kind of $f$ is automatically polarized).
One may generalize it to the case of arbitrary Picard number.
Below is a partial solution to it.

\begin{corollary}\label{smooth-cor} Let $X$ be a rationally connected smooth projective variety and $D\subset X$ a reduced divisor.
Suppose $f:X\to X$ is a polarized endomorphism such that $D$ is $f^{-1}$-invariant and $f|_{X\backslash D}$ is \'etale.
Then $(X,D)$ is a toric pair.
\end{corollary}

We say that a normal projective variety $X$ is of {\it Fano type} if there is a Weil $\Q$-divisor $\Delta$ such that the pair $(X,\Delta)$ has only klt singularities and $-(K_X+\Delta)$ is an ample $\Q$-Cartier divisor.
The assumption below of $X$ being of Fano type is necessary,
since a normal projective toric variety is known to be of Fano type.

\begin{corollary} (cf.~Remark~\ref{P_quot})\label{toric-lift-cor}
Let $f:X\to X$ be a polarized endomorphism of a normal projective variety $X$ of Fano type which is smooth in codimension $2$.
Let $D\subset X$ be an $f^{-1}$-invariant reduced divisor such that $f|_{X\backslash D}$ is quasi-\'etale and $K_X+D$ is $\Q$-Cartier.
Then there exist a quasi-\'etale cover $\pi:\widetilde{X}\to X$ and a polarized endomorphism $\widetilde{f}:\widetilde{X}\to \widetilde{X}$ such that
\begin{itemize}
\item[(1)] the endomorphism $f$ lifts to $\tilde{f}$, i.e., $\pi\circ \widetilde{f}=f\circ \pi$, and
\item[(2)] the pair $(\widetilde{X},\widetilde{D})$ is toric, where $\widetilde{D}=\pi^{-1}(D)$.
\end{itemize}
\end{corollary}

The following example satisfies the conditions of both Theorems \ref{thm-torus} and \ref{thm-pitrivial}.

\begin{example}\label{eg1}
{\rm
Let $X := \mathbb{P}^n$ and
$$f: X \, \to \, X; \,\, [X_0: \cdots : X_n] \, \mapsto \, [X_0^q : \cdots : X_n^q]$$
the power map for some $q \ge 2$. Then $f$ is polarized with $\deg f = q^n$.
Let the algebraic torus $T = (k^*)^n$ act on $X$ naturally:
$T \times X \, \to \, X$  via
$$((t_1, \dots, t_n), [X_0 : \cdots : X_n]) \, \mapsto \, [X_0 : t_1 X_1 : \cdots : t_n X_n].$$
Let $D_i := \{X_i = 0\} \subseteq X$.
Denote by $D := \sum_{i=0}^n D_i$ and $U := X \setminus D$.
Let
$$\varphi : T \, \to \, T; \,\, (t_1, \dots, t_n) \, \mapsto \, (t_1^q, \dots, t_n^q)$$
which is a surjective homomorphism.

One can check easily that $X$ is $T$-almost homogeneous with $U$ the big open orbit and $f$ is $T$-equivariant
in the sense that $f \circ g = \varphi(g) \circ f$ for all $g$ in $T$.
Hence the conditions in Theorem \ref{thm-torus} are all satisfied.
Of course, $(X, D)$ is a toric pair.

Note that $D$ is $f^{-1}$-invariant, the restriction $f|_{X \setminus D}$ is \'etale, and $\pi_1(X)$ is trivial.
So the conditions in Theorem \ref{thm-pitrivial} are all satisfied.
Clearly, we have $c(X, D) \le n + 1 - (n+1) = 0$.

One may take the toric blowups of $X$ to get more examples satisfying all conditions in Theorems
\ref{thm-torus} and \ref{thm-pitrivial}.
}
\end{example}

\begin{remark}\label{P_quot}
In Corollary \ref{toric-lift-cor}, it is not always possible to take $\pi : \widetilde{X} \to X$
to be the identity map. In other words, the pair $(X,D)$ itself may not be toric.

Indeed,
let $\widetilde{f}:\widetilde{X}=\mathbb{P}^n\to \widetilde{X}$ be the power map
of degree $q^n$ for some $q \ge 2$ as defined in Example \ref{eg1}.
The symmetric group $S_{n+1}$ in $(n+1)$-letters acts naturally on $\mathbb{P}^{n}$ as (coordinates) permutations.
Let $\widetilde{D}_i:=\{X_i = 0\} \, \subset \, \widetilde{X}$ and $\widetilde{D}:=\sum \widetilde{D}_i$.
Then $S_{n+1}$ fixes $\widetilde{D}$ (as a set).
Choose a non-trivial subgroup $G \le S_{n+1}$ such that
\begin{itemize}
\item[(i)]
the group $G$ has no non-trivial pseudo-reflections (i.e. each non-trivial $g\in G$ fixes at most a codimension-$2$ subset) and hence the quotient map
$\pi : \widetilde{X} \, \to \, X:=\mathbb{P}^{n}/G$ is quasi-\'etale; and
\item[(ii)]
the variety $X$ has only terminal singularities and hence is smooth in codimension $2$.
\end{itemize}

For instance, take $n>2$ and $G = \langle (1, 2, \dots, n+1) \rangle \cong \Z/(n+1)\Z$;
see \cite[Lemma 3]{KX}.

Note that $g\circ \widetilde{f}=\widetilde{f}\circ g$ for any $g$ in $G$.
Hence $\widetilde{f}$ descends to a polarized endomorphism $f$ on $X$.
Let $D:=\pi(\widetilde{D})$. Then we have
$\pi^{-1}(D) = \widetilde{D}$ and $f^{-1}(D)=D$.

We now check that $f$ and the pair $(X, D)$ satisfy the assumptions of Corollary \ref{toric-lift-cor}.
Since both $\widetilde{f}|_{\widetilde{X}\backslash\widetilde{D}}$ and $\pi$ are quasi-\'etale, so is $f|_{X\backslash D}$.
Clearly, $X$ is $\Q$-factorial.
Since $K_{\widetilde{X}}=\pi^*K_X$ is anti-ample, so is $K_X$. Thus,
$X$ is a Fano variety and hence of Fano type.

The pair $(X, D)$ is not toric because the number of the irreducible components of $D$ is less than $n+1$, noting that
$G$ permutes the $\widetilde{D}_i$ non-trivially; see Remark \ref{rmk-comp2}.
Of course, its quasi-\'etale cover $(\widetilde{X}, \widetilde{D})$
is a toric pair.
\end{remark}

\par \noindent
{\bf Acknowledgement.}
The second author thanks Mircea Mustata for
valuable discussions and warm
hospitality during his visit to Univ. of Michigan in December 2016;
he is also supported by an ARF of National University of Singapore.
The authors thank the referee for suggestions to improve the paper.

\section{Preliminary results}
\begin{setup}{\bf Notation and terminology}\label{nat}.

{\rm
Let $X$ be a normal projective variety of dimension $n \ge 1$.
Define:
\begin{itemize}
\item[(1)]
$q(X)=h^1(X,\mathcal{O}_X)=\dim H^1(X,\mathcal{O}_X)$ (the irregularity); and
\item[(2)]
$\widetilde{q}(X)=q(\widetilde{X})$ with $\widetilde{X}$ a smooth projective model of $X$.
\end{itemize}
}

Let $D_1$ and $D_2$ be two Cartier $\R$-divisors on $X$.
Denote by $D_1\equiv D_2$ if $D_1$ is {\it numerically equivalent} to $D_2$, i.e., if
$(D_1 - D_2)\cdot C  = 0$ for every curve $C$  on $X$.


Denote by $X_{\reg}$ the {\it smooth locus} of $X$.
Let $U\subseteq X_{\reg}$ be an open dense subset. Let $\pi:Y\to X$ be a log resolution such that $\pi$ is isomorphic over $U$ and $B=Y \setminus \pi^{-1}(U)$ is a simple normal crossing (SNC) divisor.
Define the {\it log Kodaira dimension} of $U$ as $\bar{\kappa}(U)=\kappa(Y,K_Y+B)$
(Iitaka's $D$-dimension),
which is independent of the choice of the compactification $Y$ of $U$.

Given a reduced divisor $D$ on $X$, we define the sheaf $\hat{\Omega}_X^1(\log D)$ as follows.
Let $U\subseteq X$ be an open subset with $\Codim(X-U)\ge 2$ and $D\cap U$ being a normal crossing divisor.
Denote by $\Omega_U^1(\log D\cap U)$ the locally free sheaf of germs of logarithmic 1-forms on $U$ with poles only along $U \setminus D$.
Using the open immersion $j : U\hookrightarrow X$, we define
$$\hat{\Omega}_X^1(\log D)=j_*\Omega_U^1(\log D\cap U).$$
This is a reflexive coherent sheaf on $X$.

Throughout this paper, for a pair $(X,\Delta)$, the coefficients of $\Delta$ lie in $[0,1]$.
\end{setup}

The result below is frequently used and part of \cite[Proposition 2.9]{MZ}.

\begin{lemma}\label{lem-eig} Let $f:X\to X$ be a polarized endomorphism of a normal projective variety $X$. Suppose $f^*M\equiv M$ (numerical equivalence) for some Cartier $\R$-divisor $M$. Then $M\equiv 0$.
\end{lemma}

\section{Proof of Theorem \ref{thm-torus}}

\begin{lemma}\label{toric} Let $X$ be a normal projective variety with an algebraic torus $T$-action.
Suppose $T$ has a
 Zariski-dense open orbit $U$ in $X$. Then $X$ is a toric variety.
\end{lemma}

\begin{proof} Let $x\in U$. Then we have $U=T/T_x$, where $T_x:=\{t\in T\, |\, tx=x\}$. Since $T$ is a torus, $U$ is again a torus.
For any $y\in U$ and $t\in T_x$, we have $y=t_y x$ for some $t_y\in T$.
Then $ty=tt_yx=t_ytx=y$. In particular, $T_x$ acts trivially on $U$ and hence on $X$. So the natural action of $U=T/T_x$ on itself extends to $X$.
\end{proof}

\begin{lemma}\label{logkod} Let $X$ be a normal projective variety and $U\subseteq X$ an open dense
subset such that $U$ is contained in the smooth locus of $X$. Let $D$ be the sum of all the prime divisors $D_i$ contained in $X \setminus U$. Assume $(X,D)$ is log canonical.
Let $\pi:Y\to X$ be a resolution such that $\pi$ is isomorphic over $U$ and
$Y \setminus \pi^{-1}(U)$ is an SNC divisor. Let $D'$ be the strict transform of $D$ and $E$ the sum of $\pi$-exceptional prime divisors $E_j$ such that $Y \setminus \pi^{-1}(U) = D' + E$.
Then the log Kodaira dimension $\bar{\kappa}(U) := \kappa(Y, D' + E)$ equals $\kappa(X,K_X+D)$, and
$K_Y + D' + E$ is pseudo-effective (i.e., the limit of effective $\R$-Cartier divisors) if and only if so is $K_X + D$.
\end{lemma}

\begin{proof}
Since $(X,D)$ is log canonical,
$K_Y+D'+E=\pi^\ast(K_X+D)+\sum\limits_{j} b_j E_j$ with $b_j\ge 0$.
Hence the lemma follows by the projection formula.
\end{proof}

\begin{proof}[Proof of Theorem \ref{thm-torus}]

Write $\deg f=q^{\dim(X)}$ with $q>1$. We may assume $G$ is connected and acts faithfully on $X$.
Let $x_0\in U$.
Then $U = Gx_0$ and $f(U) = f(Gx_0) = \varphi(G)f(x_0) = G f(x_0)$.
So $f(U)$ is an open dense $G$-orbit in $X$ and hence $f(U) = U$.
Further, we claim that $f^{-1}(U)=U$. Indeed,
for any $x\not\in U$, the orbit $Gx$ has $\dim (Gx) < \dim (X)$. Since
$Gf(x) = \varphi(G) f(x) = f(Gx)$, we have $\dim(Gf(x))=\dim(Gx)<\dim(X)$.
In particular, $f(x)$ is not in $U$.
So the claim is proved.
Let $D$ be the divisorial part of $X \setminus U$.
Then $f^{-1}(D) = D$.

Since $U$ is $G$-transitive,
$f|_U: f^{-1}(U) = U \to U$ is \'etale.
Thus, by the logarithmic ramification divisor formula,
we have $K_X + D = f^*(K_X + D)$.
Hence, $K_X+D\equiv 0$ by Lemma \ref{lem-eig}.
Since $K_X+D$ is $\Q$-Cartier, the pair $(X,D)$ is log canonical by \cite[Theorem 1.4]{BH}.
Let $\pi:Y\to X$ be a $G$-equivariant log resolution of $X$ such that $\pi|_{\pi^{-1}(U)}$ is isomorphic and $Y\backslash\pi^{-1}(U)$ is an SNC divisor (cf.~\cite{Ko}).
Applying Lemma \ref{logkod} and using the notation there, we have that $K_Y + D' + E$ is pseudo-effective.
Note that $G$ is linear, $\pi^{-1}(U)$ is an open dense $G$-orbit in $Y$, and $G$ acts faithfully on $Y$.
By \cite[Theorem 1.1]{BZ}, $G$ is an algebraic torus.
So $X$ is a toric variety by Lemma \ref{toric}. Since the big torus $U$ is an affine variety,
$X \setminus U$ is of pure codimension one and hence equal to $D$.
Thus $(X, D)$ is a toric pair.
\end{proof}

\section{The Complexity}\label{sec-comp}

\begin{setup}{\bf Some notation.}\label{setup2}
{\rm
Let $X$ be an $n$-dimensional normal projective variety and $D=\sum\limits_{i=1}^{\ell}D_i$ a reduced divisor on $X$. Denote by $\ell(D):=\ell$, the {\it number} of irreducible components in $D$;
and $r(D)$ the {\it rank} of the vector space spanned by $D_1,\cdots, D_\ell$ in the space of Weil $\R$-divisors modulo algebraic equivalence.

Let $(X,\Delta)$ be a log pair.
Write $\Delta = \sum_i \, a_i D_i$ with $a_i>0$ and $D_i$ distinct irreducible divisors.
Denote by
$$\langle\Delta\rangle:=\llcorner \Delta \lrcorner+\ulcorner 2\Delta\urcorner-\llcorner 2\Delta\lrcorner = \sum_{i: a_i > 1/2} \, D_i .$$
}
\end{setup}

\begin{definition}\label{def-comp}
A {\it decomposition} of $\Delta$ is an expression of the
form
$$\sum a_i S_i\le \Delta,$$
where $S_i \ge 0$ are $\Z$-divisors and $a_i \ge 0$, $1 \le i \le k$.
The {\it complexity} of this decomposition is $n+r-d$, where $r$ is the rank of the vector space
spanned by $S_1, S_2,\cdots, S_k$ in the space of Weil $\R$-divisors modulo algebraic
equivalence and $d=\sum a_i$.
The {\it complexity} $c = c(X,\Delta)$ of $(X,\Delta)$ is the infimum of the complexity of any decomposition of $\Delta$.
\end{definition}

The following theorem gives a geometric characterization of toric varieties involving the complexity by Brown, ${\rm M}^c$Kernan, Svaldi and Zong.

\begin{theorem} (cf.~\cite[Theorem 1.2]{BMSZ}) \label{thm-bmsz}
Let $X$ be a proper variety of dimension $n$ and
let $(X,\Delta)$ be a log canonical pair such that $-(K_X + \Delta)$ is nef.
If $\sum a_iS_i$ is a decomposition of complexity $c$ less than one then there is a divisor $D$ such that $(X,D)$ is a toric pair, where $D \ge \langle\Delta\rangle$ and all
but $\llcorner 2c \lrcorner$ components of $D$ are elements of the set $\{S_i\, |\, 1 \le i \le k \}$.
\end{theorem}

\begin{remark}\label{rmk-comp}
(1) If the $\Delta$ in Theorem \ref{thm-bmsz} is a reduced divisor with
the complexity $n+r(\Delta)-\ell(\Delta)\le 0$,
then Theorem \ref{thm-bmsz} implies that $(X,\Delta)$ itself is a toric pair.

(2) Let $f:X\to X$ be a polarized endomorphism and let $D$ be an $f^{-1}$-invariant reduced divisor
such that $f|_{X\backslash D}$ is quasi-\'etale and $K_X+D$ is $\Q$-Cartier. Then $K_X+D \equiv 0$ by the logarithmic ramification divisor formula and Lemma \ref{lem-eig}.
In particular, $-(K_X+D)$ is nef; further, the pair $(X, D)$ is log canonical (cf.~\cite[Theorem 1.4]{BH}).
\end{remark}

The following theorem provides us with a useful upper bound of the complexity.

\begin{theorem}\label{thm-complexity}
Let $X$ be a normal projective variety and $D$ a reduced divisor of $X$. Then,
in notation of \ref{nat} and \ref{setup2}, 
$$\ell(D)\ge h^0(X,\hat{\Omega}_X^1(\log D))+r(D)-\widetilde{q}(X) .$$
In particular, 
$$c(X,D)\le n+\widetilde{q}(X)-h^0(X,\hat{\Omega}_X^1(\log D)) .$$
\end{theorem}
\begin{proof} Let $\pi:\widetilde{X}\to X$ be a log resolution of the pair $(X,D)$ with $E$ being the reduced $\pi$-exceptional divisor.
Denote by $\widetilde{D}$ the largest reduced divisor contained in $\Supp \pi^{-1}(\text{non-klt locus}$ of $(X, D)$).
Let $D'$ be the strict transform of $D$ and $\widetilde{D}_i$ the irreducible component of $\tilde{D}$.
Then we may write $\widetilde{D}=\sum_i \widetilde{D}_i=D'+E$.

From the exact sequence
$$0\to \Omega_{\widetilde{X}}^1\to \Omega_{\widetilde{X}}^1(\log \widetilde{D})\to \oplus_i \mathcal{O}_{\widetilde{D}_i}\to 0,$$
we get
$$0\to H^0(\widetilde{X},\Omega_{\widetilde{X}}^1)\to H^0(\widetilde{X},\Omega_{\widetilde{X}}^1(\log \widetilde{D}))\to \oplus_i H^0(\widetilde{D}_i,\mathcal{O}_{\widetilde{D}_i})\to H^1(\widetilde{X},\Omega_{\widetilde{X}}^1),$$
where the connecting homomorphism essentially sends a generator 1 of $\mathcal{O}_{\widetilde{D}_i}$ for each component $\widetilde{D}_i$ to the first Chern class $c_1(\widetilde{D}_i)$.
So $\ell(\widetilde{D})=h^0(\widetilde{X},\Omega_{\widetilde{X}}^1(\log \widetilde{D}))+r(\widetilde{D})-q(\widetilde{X})$.

By \cite[Theorem 1.5]{GKKP}, $\mathcal{F}:=\pi_*\Omega_{\widetilde{X}}^1(\log \widetilde{D})$ is reflexive. Note that $\mathcal{F}|_U=\Omega_U^1(\log D\cap U)$, where $U\subseteq X_{\reg}$ such that $\Codim(X\backslash U)\ge 2$ and $D\cap U$ is a smooth divisor.
Then we have $H^0(\widetilde{X},\Omega_{\widetilde{X}}^1(\log \widetilde{D}))=H^0(X,\mathcal{F})=H^0(U,\mathcal{F}|_U)=H^0(X,\hat{\Omega}_X^1(\log D))$.

By the negativity lemma (cf.~\cite[Lemma 3.6.2]{BCHM}), $r(E)=\ell(E)$ and hence $r(D)\le r(\widetilde{D})-\ell(E)$. So the theorem is proved.
\end{proof}

\begin{remark}\label{rmk-comp2} In Theorem \ref{thm-complexity}, if $X$ is assumed to be $\Q$-factorial, then the negativity lemma implies $r(D)= r(\widetilde{D})-\ell(E)$ at the end of the proof. In particular, we will have $\ell(D)= h^0(X,\hat{\Omega}_X^1(\log D))+r(D)-\widetilde{q}(X)$ and $c(X,D)= n+\widetilde{q}(X)-h^0(X,\hat{\Omega}_X^1(\log D))$.

If $(X,D)$ is assumed to be a normal projective toric pair,
then it is known that $\hat{\Omega}_X^1(\log D)$ is free; see \cite[4.3, page 87]{Ful}. Since $X$ is rationally connected, $\widetilde{q}(X)=0$. Therefore, $\ell(D)\ge n+r(D)$
(with equality holding true when $X$ is $\Q$-factorial).
\end{remark}

\section{Proof of Theorem \ref{thm-pitrivial}}

\begin{lemma}\label{uqc} Let $X$ be a normal projective variety with
finite algebraic fundamental group $\pi_1^{\alg}(X_{\reg})$.
Then $X$ admits a universal quasi-\'etale cover $\pi:\widetilde{X}\to X$, such that $\pi_1^{\alg}(\widetilde{X}_{\reg})$ is trivial and any surjective endomorphism $f$ of $X$ lifts to $\widetilde{X}$.
\end{lemma}

\begin{proof} Since $\pi_1^{\alg}(X_{\reg})$ is finite, there is a universal quasi-\'etale cover $\pi:\widetilde{X}\to X$ such that $\pi_1^{\alg}(\widetilde{X}_{\reg})$ is trivial.
Let $W$ be the normalization of the fibre product of $f$ and $\pi$ and $W_0$ a dominant irreducible component of $W$. Then $W_0\to X$ is also a quasi-\'etale cover. Taking the universal quasi-\'etale cover of $W_0$ which is $\widetilde{X}$, we are done.
\end{proof}

The same argument of \cite[Proposition 2.4]{HN} gives the following.

\begin{proposition}\label{codim3-prop} Let $X$ be a normal projective variety smooth in codimension $2$ and $D\subset X$ a reduced divisor.
Suppose $f:X\to X$ is a polarized endomorphism such that $f^{-1}(D)=D$ and $f|_{X\backslash D}$ is quasi-\'etale. Then there is a smooth open subset $U\subseteq X$ such that $D \cap U$ is a normal crossing divisor and $\Codim(X \backslash U) \ge 3$.
In particular, $\hat{\Omega}_X^1(\log D)$ is locally free over $U$.
\end{proposition}

The following slightly extends \cite[Propositions 2.2 and 2.3]{HN}.

\begin{proposition}\label{mix-prop} Let $X$ be a normal projective variety
which is of dimension $n \ge 2$ and smooth in codimension $2$, and $D\subset X$ a reduced divisor.
Suppose $f:X\to X$ is a polarized endomorphism such that $f^{-1}(D)=D$ and $f|_{X\backslash D}$ is quasi-\'etale.
Let $H$ be an ample divisor on $X$ such that $f^*H\sim qH$ for some $q>1$. Then the following hold.
\begin{itemize}
\item[(1)] The intersection numbers vanish: $c_1(\hat{\Omega}_X^1(\log D))\cdot H^{n-1}= c_1(\hat{\Omega}_X^1(\log D))^2\cdot H^{n-2}= c_2(\hat{\Omega}_X^1(\log D))\cdot H^{n-2}= 0$.
\item[(2)] The sheaf $\hat{\Omega}_X^1(\log D)$ is reflexive and $H$-slope semistable.
\end{itemize}
\end{proposition}

\begin{proof} By Proposition \ref{codim3-prop}, there is a smooth open subset $U\subseteq X$ such that $D \cap U$ is a normal crossing divisor and $\Codim(X \backslash U) \ge 3$. Since $f|_{X\backslash D}$ is quasi-\'etale, $f|_{f^{-1}(U)\backslash D}$ is \'etale by the purity of branch loci.

There is a natural morphism $\varphi:f^*\hat{\Omega}_X^1(\log D)\to \hat{\Omega}_X^1(\log D)$ and $\varphi|_{f^{-1}(U)}$ is an isomorphism.
So for $1\le i\le 2$, we have $$f^*c_i(\hat{\Omega}_X^1(\log D))=c_i(f^*\hat{\Omega}_X^1(\log D))= c_i(\hat{\Omega}_X^1(\log D)).$$
Then $$q^{n-i}c_i(\hat{\Omega}_X^1(\log D))\cdot H^{n-i}=f^*c_i(\hat{\Omega}_X^1(\log D)\cdot (f^*H)^{n-i}=(\deg f)c_i(\hat{\Omega}_X^1(\log D)\cdot H^{n-i}$$
implies $c_i(\hat{\Omega}_X^1(\log D))\cdot H^{n-i}= 0$ for $1\le i\le 2$,
noting that $\deg f = q^n > 1$.
The proof for $c_1(\hat{\Omega}_X^1(\log D))^2\cdot H^{n-2}=0$, is similar.
%

For (2), suppose the contrary that $\hat{\Omega}_X^1(\log D)$ is not $H$-slope semistable.
Then there is a coherent subsheaf $\mathcal{F}\subset \hat{\Omega}_X^1(\log D)$ such that
$$\mu_H(\mathcal{F}):=\frac{c_1(\mathcal{F})\cdot H^{n-1}}{\rank \mathcal{F}}>0.$$
Note that $$s=\sup\{\mu_H(\mathcal{F})\,|\, \mathcal{F}\subset \hat{\Omega}_X^1(\log D)\}<\infty.$$
So for some $k\gg 1$, $\mu_H(g^*\mathcal{F})=q^k\mu_H(\mathcal{F})>s$ with $g=f^k$.
Let $i:f^{-1}(U)\hookrightarrow X$ be the inclusion map and let $\mathcal{G}:=i_*((g^*\mathcal{F})|_{g^{-1}(U)})$.
Then $\mu_H(\mathcal{G})=\mu_H(g^*\mathcal{F})>s$.
Note that $(g^*\mathcal{F})|_{g^{-1}(U)}$ is a subsheaf of the locally free sheaf $(g^*\hat{\Omega}_X^1(\log D))|_{g^{-1}(U)}\cong \hat{\Omega}_X^1(\log D)|_{g^{-1}(U)}$.
Since $\Codim(X \backslash g^{-1}(U)) \ge 3$ and $i_*$ is left exact,
$\mathcal{G}$ is a coherent subsheaf of the reflexive sheaf $\hat{\Omega}_X^1(\log D)$ .
So we get a contradiction and (2) is proved.
\end{proof}

\begin{theorem}\label{HN-thm}
Let $X$ be a normal projective variety
which is smooth in codimension $2$, and $D\subset X$ a reduced divisor such that:
\begin{itemize}
\item[(i)]
there is a Weil $\Q$-divisor $\Gamma$ such that the pair $(X, \Gamma)$ has only klt singularities;
\item[(ii)]
there is a polarized endomorphism $f:X\to X$ such that $f^{-1}(D)=D$ and $f|_{X\backslash D}$ is quasi-\'etale; and \item[(iii)]
the algebraic fundamental group $\pi_1^{\alg}(X_{\reg})$ of the smooth locus $X_{\reg}$ of $X$ is trivial.
\end{itemize}
Then $\hat{\Omega}_X^1(\log D)$ is free, i.e., isomorphic to $\mathcal{O}_X^{\oplus n}$,
where $n := \dim X$.
\end{theorem}

\begin{proof} The case $n = 1$ is clear. If $n \ge 2$, apply Proposition \ref{mix-prop} and \cite[Theorem 1.20]{GKP}.
\end{proof}

\begin{proof}[Proof of Theorem \ref{thm-pitrivial}]
If $\dim(X)=1$, then $X\cong\mathbb{P}^1$ since $\pi_1^{\alg}(X_{\reg})$ is trivial and $\deg f>1$.
Clearly, $\ell(D)=\deg D=2$ and $r(D)=1$; see Remark \ref{rmk-comp}. So $c(X,D)=0$.
Now we may assume $\dim(X)>1$.
By Theorem \ref{HN-thm}, $\hat{\Omega}_{X}^1(\log D)$ is free.
Since $(X,\Gamma)$ is klt, $X$ has rational singularities.
So we have $\widetilde{q}(X) = q(X) = 0$ by the assumption.
Thus $c(X,D)\le\widetilde{q}(X)=0$ (cf.~Theorem \ref{thm-complexity}).
\end{proof}

\begin{proof}[Proof of Corollary \ref{smooth-cor}]
By Theorem \ref{thm-pitrivial}, $c(X,D)\le 0$. So $(X,D)$ is a toric pair by Remark \ref{rmk-comp} and \cite[Theorem 1.2]{BMSZ} or Theorem \ref{thm-bmsz}.
\end{proof}

\begin{proof}[Proof of Corollary \ref{toric-lift-cor}]
Since $X$ is of Fano type, there is a Weil $\Q$-divisor $\Delta$ such that
the pair $(X,\Delta)$ is klt and $-(K_X+\Delta)$ is ample.
By \cite[Theorem 1.13]{GKP}, $\pi_1^{\alg}(X_{\reg})$ is finite.
By Lemma \ref{uqc}, $X$ admits a universal quasi-\'etale cover $\pi:\widetilde{X}\to X$, such that $\pi_1^{\alg}(\widetilde{X}_{\reg})$ is trivial and $f$ lifts to $\widetilde{f}$ on $\widetilde{X}$.
Note that $\widetilde{X}$, like $X$, is still smooth in codimension $2$ by the purity of branch loci.

Let $\widetilde{D}:=\pi^{-1}(D)$ and $\widetilde{\Delta}:=\pi^*(\Delta)$.
Then $K_{\widetilde{X}}+\widetilde{D}$ equals $\pi^*(K_X+D)$ and hence is $\Q$-Cartier
(and numerically trivial; see Remark \ref{rmk-comp}).
Also $K_{\widetilde{X}}+\widetilde{\Delta}$ equals $\pi^*(K_X+\Delta)$ and is hence anti-ample.
Note that $(\widetilde{X}, \widetilde{\Delta})$ is also klt (cf.~\cite[Proposition 5.20]{KM}).
Therefore, $\widetilde{X}$ is also of Fano type and $q(\widetilde{X}) = 0$ by the
Kawamata-Viehweg vanishing.
Since $\widetilde{f}$ is the lifting of $f$, it is polarized and $\widetilde{f}^{-1}(\widetilde{D})=\widetilde{D}$.
Since both $\pi:\widetilde{X}\to X$ and $f|_{X \setminus D}$ are quasi-\'etale,
so is $\widetilde{f}|_{\widetilde{X}\backslash \widetilde{D}}$.

Thus Theorem \ref{thm-pitrivial} is applicable: $c(\widetilde{X},\widetilde{D})\le 0$. So $(\widetilde{X},\widetilde{D})$ is a toric pair by Remark \ref{rmk-comp} and \cite[Theorem 1.2]{BMSZ} or Theorem \ref{thm-bmsz}.
\end{proof}


\begin{thebibliography}{99}
\bibitem{BCHM}
C. Birkar, P. Cascini, C. D. Hacon and J.~${\rm M}^c$Kernan,
Existence of minimal models for varieties of log general type, J. Amer. Math. Soc., \textbf{23}(2):405-468, 2010.

\bibitem{BZ}
M.~Brion and D. -Q.~Zhang,
Log Kodaira dimension of homogeneous varieties,
Algebraic varieties and automorphism groups, 1-6, 
Adv. Stud. Pure Math., \textbf{75}, Math. Soc. Japan, Tokyo, 2017. 

\bibitem{BH}
A.~Broustet and A.~H\"oring,
Singularities of varieties admitting an endomorphism,
Math. Ann. \textbf{360} (2014), no. 1-2, 439-456.

\bibitem{BMSZ}
M.~Brown, J.~${\rm M}^c$Kernan, R.~Svaldi and H.~Zong,
A geometric characterisation of toric varieties,
Duke Math. J. Volume \textbf{167}, Number 5 (2018), 923-968.

\bibitem{Ful}
W.~Fulton,
Introduction to toric varieties,
No. \textbf{131}. Princeton University Press, 1993.
\bibitem{GKKP}
D.~Greb, S.~Kebekus, S.~J.~Kov\'acs and T.~Peternell,
Differential forms on log canonical spaces,
Publ. Math. Inst. Hautes \'Etudes Sci. No. \textbf{114} (2011), 87-169.

\bibitem{GKP}
D.~Greb, S.~Kebekus and T.~Peternell,
\'Etale fundamental groups of Kawamata log terminal spaces, flat sheaves, and quotients of Abelian varieties,
Duke Math. J. \textbf{165} (2016), no. 10, 1965-2004.

\bibitem{HN}
J.~M.~Hwang and N.~Nakayama, On endomorphisms of Fano manifolds of Picard number one, Pure Appl. Math. Q, \textbf{7}(4), pp.1407-1426, 2011.

\bibitem{Ko}
J.~Koll\'ar,
Lectures on resolution of singularities,
Ann. of Math. Stud. \textbf{166}, Princeton Univ. Press, 2007.

\bibitem{KM}
J.~Koll\'ar and S.~Mori,
Birational geometry of algebraic varieties,
Cambridge Tracts in Math.,
\textbf{134} Cambridge Univ. Press, 1998.

\bibitem{KX}
J.~Koll\'ar and C.~Xu,
Fano varieties with large degree endomorphisms,
\href{http://arXiv.org/abs/0901.1692}{\tt arXiv:0901.1692}.

\bibitem{MZ}
S.~Meng and D. -Q.~Zhang,
Building blocks of polarized endomorphisms of normal projective varieties,
Adv. Math. \textbf{325} (2018), 243-273.

\bibitem{Sho}
V.~V.~Shokurov,
Complements on surfaces, J. Math. Sci. (New York) \textbf{102} (2000), no. 2,
3876-3932, Algebraic geometry, \textbf{10}.

\bibitem{Zsw}
S. W.~Zhang,
Distributions in algebraic dynamics,
Surveys in differential geometry, Vol. X, 381-430, Surv. Differ. Geom., \textbf{10}, Int. Press, Somerville, MA, 2006.
\end{thebibliography}
\end{document}